\documentclass[11pt]{amsart}

\usepackage{mathrsfs}
\usepackage{latexsym}
\usepackage{amsmath,amsthm, amsfonts, amssymb,amsgen}
\usepackage{cite}
\usepackage{enumerate}

\makeatletter
\DeclareFontFamily{OMX}{MnSymbolE}{}
\DeclareSymbolFont{MnLargeSymbols}{OMX}{MnSymbolE}{m}{n}
\SetSymbolFont{MnLargeSymbols}{bold}{OMX}{MnSymbolE}{b}{n}
\DeclareFontShape{OMX}{MnSymbolE}{m}{n}{
    <-6>  MnSymbolE5
   <6-7>  MnSymbolE6
   <7-8>  MnSymbolE7
   <8-9>  MnSymbolE8
   <9-10> MnSymbolE9
  <10-12> MnSymbolE10
  <12->   MnSymbolE12
}{}
\DeclareFontShape{OMX}{MnSymbolE}{b}{n}{
    <-6>  MnSymbolE-Bold5
   <6-7>  MnSymbolE-Bold6
   <7-8>  MnSymbolE-Bold7
   <8-9>  MnSymbolE-Bold8
   <9-10> MnSymbolE-Bold9
  <10-12> MnSymbolE-Bold10
  <12->   MnSymbolE-Bold12
}{}

\let\llangle\@undefined
\let\rrangle\@undefined
\DeclareMathDelimiter{\llangle}{\mathopen}%
                     {MnLargeSymbols}{'164}{MnLargeSymbols}{'164}
\DeclareMathDelimiter{\rrangle}{\mathclose}%
                     {MnLargeSymbols}{'171}{MnLargeSymbols}{'171}
\makeatother

\newtheorem{theorem}{Theorem}[section]
\newtheorem{lemma}[theorem]{Lemma}

\newtheorem{cor}[theorem]{Corollary}
{\theoremstyle{remark}
\newtheorem{remark}[theorem]{Remark}}

{\theoremstyle{remark}}
{\theoremstyle{definition}
\newtheorem{definition}[theorem]{Definition}}
{\theoremstyle{remark}
\newtheorem{claim}{Claim}}

\def\Z{\mathbb Z}

\renewcommand{\to}{\longrightarrow}
\newcommand{\ch}{\mathop{\mathrm{Ch}}\nolimits}
\newcommand{\rmv}[1]{}
\newcommand{\Irr}{\mathop{\mathrm{Irr}}\nolimits}
\numberwithin{equation}{section}

\begin{document}

\title{Character Theory of Monoids over an Arbitrary Field}

\author{Ariane M. Masuda, Luciane Quoos \and Benjamin Steinberg}
\address[A.~M.~Masuda]{Department of Mathematics \\ New York City College of Technology\\ 300 Jay Street\\ Brooklyn, NY 11201\\ USA}
\email{amasuda@citytech.cuny.edu}
\address[L. Quoos]{Instituto de Matem\'atica \\ Universidade Federal do Rio de Janeiro \\ Cidade Universit\'aria \\ Rio de Janeiro, RJ 21941-909\\ Brazil }
\email{luciane@im.ufrj.br}
\address[B. Steinberg]{Department of Mathematics\\
    City College of New York\\
    Convent Avenue at 138th Street\\
    New York, NY 10031\\
    USA}
\email{bsteinberg@ccny.cuny.edu}
\thanks{The first author was partially supported by PSC-CUNY grants (\#66133-00 44 and \#67111-00 45).
 The second author was partially supported by Siemens, project INST.MAT.-18.414. The third author was partially supported by a grant from the Simons Foundation (\#245268
to Benjamin Steinberg), the Binational Science Foundation of Israel and the US (\#2012080), and a CUNY Collaborative Incentive Research Grant. Some of this work was performed while the first, second and third authors were visiting UFRJ, CUNY and IMPA, respectively; they gratefully acknowledge the warm hospitality they received.}
\date{October 26, 2014}

\begin{abstract}
The basic character theory of finite monoids over the complex numbers was developed in the sixties and seventies based on work of Munn, Ponizovsky, McAlister, Rhodes and Zalcstein.  In particular, McAlister determined the space of functions spanned by the irreducible characters of a finite monoid over $\mathbb C$ and the ring of virtual characters.  In this paper, we present the corresponding results over an arbitrary field.

As a consequence, we obtain a quick proof of the theorem of Berstel and Reutenauer that the characteristic function of a regular cyclic language is a virtual character of the free monoid. This is a crucial ingredient in their proof of the rationality of the zeta function of a sofic shift in symbolic dynamics.
\end{abstract}
\maketitle

\section{Introduction}
The representation theory of finite monoids has enjoyed a rebirth over the past two decades~\cite{Putcharep5,Putcharep3,Putcharep1,Putcharep4,Putcharep2,gmsrep,mobius1,mobius2,Saliola,rrbg,Jtrivialpaper,bihecke,DO,MSS,BergeronSaliola}, in a large part due to applications to Markov chain theory~\cite{BHR,Brown1,Brown2,DiaconisBrown1,DiaconisAthan,DiaconisICM,AyyerKleeSchilling,bjorner1,bjorner2,Saliolaeigen,GrahamChung,ayyer_klee_schilling.2013,ayyer_schilling_steinberg_thiery.sandpile.2013,ayyer_schilling_steinberg_thiery.2013}, data analysis~\cite{Malandro1,Malandro2,Malandro3} and automata theory~\cite{AMSV,mortality,transformations}.  This paper is a further contribution to the representation theory of finite monoids.

We say that two elements $a$ and $b$ of a finite monoid $M$ are character equivalent over a field $K$ if every irreducible character of $M$ over $K$ agrees on $a$ and $b$.  Character equivalence for the field of complex numbers was described by McAlister~\cite{McAlisterCharacter}, see also~\cite{RhodesZalc} and~\cite{Mazorchuk}. The arguments of~\cite{McAlisterCharacter}, in fact, work over any algebraically closed field of characteristic $0$.  In this paper we  determine character equivalence for an arbitrary field and we prove that the irreducible characters form a basis for the algebra of $K$-valued functions on $M$ which are constant on character equivalence classes.  As is often the case in the representation theory of finite monoids, the result is obtained via a reduction to the case of finite groups, which was essentially handled by Berman~\cite{Berman}.

Recall that a virtual character is a difference of two characters.  The virtual characters of $M$ over $K$ form a ring and we show that it is isomorphic to the direct product of the rings of virtual characters of maximal subgroups of $M$ over $K$ (one per regular $\mathcal D$-class), as is the case over $\mathbb C$~\cite{McAlisterCharacter}.  Moreover, we show that a function that is constant on character equivalence classes is a virtual character if and only if its restriction to each maximal subgroup is a virtual character.

As an application of our results, we give a character theoretic proof of a result of Berstel and Reutenauer~\cite{BerstelReutenauer}, which they put to good use in order to prove the rationality of the zeta function of a sofic shift in symbolic dynamics~\cite{MarcusandLind}. More precisely, if $A^*$ is the free monoid on a finite set $A$, then a language $L\subseteq A^*$ is said to be cyclic if:
\begin{enumerate}[i)]
\item $u\in L\iff u^n\in L$ for all $n>0$;
\item $uv\in L\iff vu\in L$.
\end{enumerate}
For example, if $\mathcal X$ is a subshift of $A^{\mathbb Z}$, then the set $L$ of all words $w$ such that $\cdots ww.ww\cdots\in \mathcal X$ is cyclic.  Recall that a language $L\subseteq A^*$ is regular if it is accepted by a finite state automaton~\cite{EilenbergA}.  Berstel and Reutenauer proved that the characteristic function $A^*\to \{0,1\}$ of a regular cyclic language $L\subseteq A^*$ is a virtual character of $A^*$ (over any field).  From this, they easily deduced the rationality of the zeta function of $L$ and hence the rationality of zeta functions of sofic shifts.

In a recent paper~\cite{Perrincomplred} Perrin gave another proof that the characteristic function of a regular cyclic language is a $K$-linear combination of irreducible characters over any algebraically closed field $K$ of characteristic $0$ using the character theory results of McAlister~\cite{McAlisterCharacter}.  This motivated us to look for a proof of the original result of Berstel and Reutenauer using the character theory of finite monoids over an arbitrary field.

\section{Characters of monoids over an arbitrary field}

Let $M$ be a monoid and $K$  a field of characteristic $p \geq 0$.  A \emph{representation} of $M$ over $K$ is a monoid homomorphism $\varphi\colon  M \to M_r(K)$.  One says that $r$ is the \emph{degree} of $\varphi$. A representation is \emph{irreducible} if there is no proper nonzero subspace $V$  of $K^r$ such that $\varphi(m)\cdot v \in V$ for all $m\in M$ and $v\in V$. Mostly we shall assume that $M$ is finite, except in Section~\ref{s:application}.

We say that two representations $\varphi_1$ and $\varphi_2$ of degree $r$ are {\emph{equivalent}}, written $\varphi_1 \sim \varphi_2$, if there exists an invertible matrix $A \in M_r(K)$ such that $A\varphi_1(m)A^{-1}=\varphi_2(m)$ for all $m \in M$.  The set of all equivalence classes of irreducible representations of $M$ over $K$ is denoted $\Irr_K(M)$.

The \emph{character} of a representation $\varphi$ is the map $\chi_\varphi\colon M \to K $ given by \[\chi_\varphi(m) =\mathrm{Tr}(\varphi(m))\] where $\mathrm{Tr}(A)$ denotes the trace of a matrix $A$. Notice that if the degree $r$ is $1$,  then the character agrees with the representation. An \emph{irreducible character} is the character of an irreducible representation.  If $\varphi_1\sim \varphi_2$, then $\chi_{\varphi_1}=\chi_{\varphi_2}$.

The following theorem is Theorem~9.22 in~\cite{Isaacs}.

\begin{theorem}\label{indepgroupchar}
If $\chi_1, \ldots, \chi_m$ are characters of inequivalent irreducible representations of a finite group $G$ over a field $K$, then they are linearly independent and, in particular, non-zero and distinct.
\end{theorem}

It is well known~\cite[Theorem~27.22]{curtis} that if $K$ is an algebraically closed field of characteristic $0$, then the number of irreducible characters of a finite group $G$ over $K$ is the number of conjugacy classes.  Moreover, the irreducible characters form a basis for the space of functions $f\colon G\to K$ constant on conjugacy classes. McAlister~\cite{McAlisterCharacter} obtained the appropriate generalization to representations of finite monoids over an algebraically closed field; the result was obtained at the same time independently by Rhodes and Zalcstein, but only published many years later~\cite{RhodesZalc};  see also~\cite{Mazorchuk}.  Our goal is to extend McAlister's results to arbitrary fields.

\begin{definition}
Let $K$ be a field and $M$ a monoid. We say that $m_1, m_2 \in M$ are {\emph{character equivalent}} if $\chi(m_1)=\chi(m_2)$ for all characters $\chi$ of $M$, or equivalently for all irreducible characters $\chi$ of $M$, over $K$.
\end{definition}

For example, if $K=\mathbb C$ and $M$ is a finite group, then character equivalence is conjugacy.

Let $M$ be a finite monoid. Then $e \in M$ is an {\emph{idempotent}} if $e^2=e$.  The set of idempotents is denoted $E(M)$.
The set $eMe=\{eme \,|\, m \in M \}=\{ m \in M \, |\, em=m=me\}$ is a monoid with identity $e$ and the multiplication induced from $M$. Let $G_e$ be the group of invertible elements of $eMe$. It is called the \emph{maximal subgroup} of $M$ at $e$.

If $m$ is an element of a finite monoid $M$, then there exist
$r, s >0$ such that $m^r=m^{r+s}.$
By choosing $r$ and $s$ minimal, the set $G=\{m^r, \dots , m^{r+s-1} \}$ is a cyclic group of order $s$ isomorphic to $\mathbb Z_s$ via the map sending $m^t$ to the residue class of $t\bmod s$~\cite[Theorem~1.9]{CP}. The identity of the group is $m^k$, where $k$ is the unique integer such that $r\leq k \leq r+s-1$ and $k \equiv 0\bmod s$, and is denoted symbolically by $m^{\omega}$. The group $G$ is generated by $m^{\omega}m$, which we denote $m^{\omega+1}$.  Note that if $M$ is a finite monoid, then $m^{|M|!}=m^{\omega}$ for all $m\in M$ because, retaining the above notation, $r,s\leq |M|$ and hence $|M|!\equiv 0\bmod s$.

\begin{definition}
Let $M$ be a finite monoid and let $p=0$ or be prime. Let $m \in M$.
\begin{enumerate}[a)]
\item We say that $m$ is a {\emph{group element}} if there exists $s >0$ such that $m=m^s$ or, equivalently, $m=m^{\omega+1}$.
\item For a group element $m$, the minimal $s>0$ such that $m=m^{s+1}$ is its \emph{order}, denoted $|m|$, in which case $\{m,\ldots, m^s\}$ is a cyclic group of order $|m|$.
\item We say $m$ is \emph{$p$-regular} if $m$ is a group element and $p=0$ or $p$ does not divide  $|m|$.
\end{enumerate}
\end{definition}

If $p>0$ is a prime and $g$ is an element of a finite group $G$ with identity $e$, then $g$ can be uniquely factored as $g=g_pg_{p'}$ where $g_p$ has order a $p$-power, $g_{p'}$ has order prime to $p$ (i.e., $g_{p'}$ is $p$-regular) and $g_p$, $g_{p'}$ commute.  Indeed, if $|g|$ is $p^kt$ with $\gcd(t,p)=1$, then $g_p=g^r$ and $g_{p'}=g^{s}$ where $r\equiv 1\bmod p^k$, $r\equiv 0\bmod t$ and $s\equiv 1\bmod t$, $s\equiv 0\bmod p^k$. See~\cite[Lemma~40.3]{curtis} and its proof for details.  If $p=0$, we define $g_p=e$ and $g_{p'}=g$.

If $M$ is a finite monoid, then for any $m \in M$, we observe that $m^{\omega+1}$ is always a group element and so $m^{\omega+1}_{p'}$ makes sense.

Let us fix some notation that will be used throughout this section.
Let $n$ be the least common multiple of the orders of the $p$-regular elements of $M$. Notice that $p \nmid n$ and that we have a homomorphism $\theta$ from the Galois group $\mathrm{Gal}(K(\xi_n)/K)$ into the multiplicative group $\Z_n^*$, where $\xi_n$ is a primitive $n^{th}$-root of unity in a fixed algebraic closure $\bar{K}$ of $K$, defined as follows. If $\sigma \in \mathrm{Gal}(K(\xi_n)/K)$, then $\sigma(\xi_n)=\xi_n^\ell,$ where $ \ell \in \Z_n^*$, and we put $\theta(\sigma)=\ell$. Let $T=\theta(\mathrm{Gal}(K(\xi_n)/K))\leq \Z_n^*$, also denoted $\mathop{\mathrm{Im}}\mathrm{Gal}(K(\xi_n)/K)$.

Let us now assume that $M$ is a group $G$.  Following Berman~\cite{Berman}, we say that $p$-regular elements $g,h\in G$ are \emph{$K$-conjugate}, denoted $g\sim_K h$, if there exist $x\in G$ and $j\in T$ such that $xgx^{-1}=h^j$.  It is easy to check that this is an equivalence relation on the set $p\text{-}\mathrm{reg}(G)$ of $p$-regular elements of $G$.

\begin{theorem}[Berman~\cite{Berman}]\label{Berman}
Let $G$ be a finite group and $K$ a field.  Then the number of equivalence classes of irreducible representations of $G$ over $K$ is the number of $K$-conjugacy classes of $p$-regular elements of $G$.
\end{theorem}

This result can be found in characteristic $0$ as~\cite[Theorem~42.8]{curtis}; a proof in characteristic $p$ using Brauer characters can be found in~\cite{ReinerPfBerman}.

We shall need the following lemma clarifying $\sim_K$.

\begin{lemma}\label{galoisres}
If $t$ is such that the orders of all $p$-regular elements of a finite group $G$ divide $t$ and $T'=\mathop{\mathrm{Im}}\mathrm{Gal} (K(\xi_t)/K)$ in $\mathbb Z_t^*$, then $g\sim_K h$ if and only if there exist $x\in G$ and $j\in T'$ such that $xgx^{-1}=h^j$.\end{lemma}
\begin{proof}
Let $n$ be the least common multiple of the orders of the $p$-regular elements of $G$ and observe that it is a divisor of $t$. Then since $K(\xi_t)/K$ is abelian, it follows that the subfield $K(\xi_n)$ is $\mathrm{Gal}(K(\xi_t)/K)$-invariant and that the restriction map $\mathrm{Gal}(K(\xi_t)/K)\to \mathrm{Gal}(K(\xi_n)/K)$ is a surjective homomorphism. Also if $\sigma\in \mathrm{Gal}(K(\xi_t)/K)$ with $\sigma(\xi_t)=\xi_t^j$, then $\sigma(\xi_n)=\xi_n^j$.   Therefore, if $\alpha\colon \mathbb{Z}_t^*\to\mathbb{Z}_n^*$ is the surjective homomorphism given by $\alpha(j)=j\bmod{n}$, then $T=\alpha(T')$.  Moreover, if $j\in T'$, then
$h^j=h^{\alpha(j)}$ because $|h|$ divides $n$, which divides $t$. The lemma follows.
\end{proof}

Recall that two idempotents $e$ and $f$ in a monoid $M$ are \emph{$\mathcal D$-equivalent}, written $e\mathrel{\mathcal D} f$, if there exist $x,y\in M$ with $xyx=x$, $yxy=y$, $xy=e$ and $yx=f$; see~\cite[Section~2.3]{CP} or~\cite[Proposition~1.3]{TilsonXI} for details.  An equivalence class for the $\mathcal D$-relation is called a \emph{$\mathcal D$-class}. It is well known that in a finite monoid $M$, one has that  $e \mathrel{\mathcal{D}}f$ if and only if $MeM=MfM$~\cite[Appendix A]{qtheor}. (In general, elements $m,n\in M$ are called \emph{$\mathcal J$-equivalent} if $MmM=MnM$.)

The following well-known lemma, cf.~\cite[Proposition~1.4]{TilsonXI} or~\cite[Theorem~2.20]{CP}, will play a key role later.

\begin{lemma}\label{conjugationlemma}
Let $M$ be  a  monoid and let $e,f\in M$ be $\mathcal{D}$-equivalent idempotents.  Suppose that $x,y\in M$ with $xyx=x$, $yxy=y$, $xy=e$ and $yx=f$.  Then $\varphi\colon eMe\to fMf$ and $\psi\colon fMf\to eMe$ defined by $\varphi(a)=yax$ and $\psi(b)=xay$ are inverse isomorphisms of monoids.  Consequently, $\varphi$ and $\psi$ restrict to inverse isomorphisms of $G_e$ and $G_f$.
\end{lemma}

Let $e_1, \dots, e_k$ be elements in $E(M)$ representing the distinct
$\mathcal{D}$-classes of $E(M)$. We can define a partial order on the set $\{e_1,\ldots, e_k\}$ by  $e_i \preceq e_j$ if and only if $Me_iM \subseteq Me_jM$.

The following is a fundamental result, proved independently by Munn~\cite{Munn1} and Ponizovsky~\cite{Poni} based on earlier work of Clifford~\cite{Clifford2}; see~\cite[Theorem~5.33]{CP} or~\cite{RhodesZalc}.  A simpler, module-theoretic approach can be found in~\cite{gmsrep}.

\begin{theorem}[Clifford-Munn-Ponizovsky]\label{indcar}
Let $M$ be a finite monoid, $K$ a field and $e_1,\ldots, e_k$ idempotents representing the $\mathcal D$-classes of $E(M)$.
If $\varphi\colon M \to M_r(K)$ is an irreducible representation, then there exists a unique minimal idempotent $e_i$ (with respect to $\preceq$) among $e_1,\ldots, e_k$ such that $\varphi(e_i)\neq 0$.  Moreover,
\begin{enumerate}[i)]
\item for $m \in M$, $\varphi(m) \neq 0$ if and only if $e_i \in MmM$;
\item $\varphi_{|_{G_{e_i}}} \sim
\begin{bmatrix}
\hat{\varphi} & 0\\
0 & 0\\
\end{bmatrix}$,
where $\hat{\varphi}$ is an irreducible representation of the group $G_{e_i}$.
\end{enumerate}
The element $e_i$ is called the \emph{apex} of $\varphi$.

In addition, the map \[\psi\colon \Irr_K(M) \to \displaystyle{\bigsqcup_{i=1}^k  \Irr_K(G_{e_i})}\] given by $\psi(\varphi)=\hat{\varphi}$ is a bijection.
\end{theorem}

Let us interpret part of the above theorem in terms of characters.

\begin{cor}\label{rescar}
Let $\varphi\colon M\to M_r(K)$ be an irreducible representation with apex $e_i$.   If $\chi$ is the character of $\varphi$, then $\chi_{|_{G_{e_i}}}=\chi_{\hat{\varphi}}$ where $\hat{\varphi}$ is as in Theorem~\ref{indcar}.
\end{cor}
\begin{proof}
This is immediate from ii) of Theorem~\ref{indcar}.
\end{proof}

We can now prove the analogue of Theorem~\ref{indepgroupchar} for monoids.

\begin{theorem}\label{indepofchar}
If $\chi_1, \ldots, \chi_s$ are characters of inequivalent irreducible representations of a finite monoid $M$ over a field $K$, then they are non-zero, distinct and linearly independent.
\end{theorem}
\begin{proof}
Without loss of generality, we may suppose that there are $c_1, \dots, c_s$ in $K\setminus \{0\}$ such that
\[0=c_1\chi_1 + \cdots + c_s\chi_s.\]

Let $e=e_i$ be  minimal with respect to $\preceq$ among the apexes of the representations $\chi_1, \dots , \chi_s$. Reordering the representations we can assume that $e$ is the apex for $\chi_1, \dots , \chi_t$ where $1\leq t\leq s$, and that $\chi_{t+1}, \dots , \chi_s$ have a different apex than $e$.

Any $g$ in the group $G_e$ satisfies $MeM=MgM$. So, if $e_\ell$ is the apex of $\chi_j$ with  $j > t$ (and hence $\ell\neq i$), then $Me_\ell M \nsubseteq MeM=MgM$ by minimality of $e$. Therefore, $\varphi_j(g)=0$ by i) of Theorem~\ref{indcar}. We conclude that $\chi_{j}(g)=0$, for all $g \in G_e$ and $t<j\leq s$.

Therefore, for any $g \in G_e$, we obtain
\[0=c_1\chi_1(g)+ \cdots + c_t\chi_t(g).\]

But ${\chi_1}_{|_{G_e}}, \dots, {\chi_t}_{|_{G_e}}$ are characters of distinct irreducible representations of $G_e$ by Theorem~\ref{indcar} and Corollary~\ref{rescar}, a contradiction with the linear independence of irreducible characters of finite groups, cf.~Theorem~\ref{indepgroupchar}.
\end{proof}

\begin{remark}
In~\cite[Theorem~2.1]{guralnick} it is shown that, for arbitrary monoids (possibly infinite), the only obstruction to characters of distinct irreducible representations being linearly independent is that one or more of the characters might be identically zero and that this can indeed happen for infinite groups.
\end{remark}

The following lemma is crucial for understanding character equivalence. We recall that if $n$ is the least common multiple of the orders of the $p$-regular elements of $M$, then $T$ denotes the image of $\mathrm{Gal}(K(\xi_n)/K)$ in $\mathbb Z_n^*$.

\begin{lemma}\label{characters}
If $\chi$ is a character of $M$ over a field $K$ of characteristic $p \geq 0$ and $a,b\in M$, then
\begin{enumerate}[a)]
\item $\chi(ab) = \chi(ba)$;
\item $\chi(a) = \chi(a_{p'}^{\omega+1})$;
\item$\chi(a)=\chi(a^j)$ if $a$ is $p$-regular and $j\in T$.
\end{enumerate}
\end{lemma}
\begin{proof}
Let $\varphi\colon M\to M_\ell(K)$ be a representation.
Since $\chi(m) = \mathrm{Tr}(\varphi(m))$, it follows that
\[\chi(ab) = \mathrm{Tr}(\varphi(ab)) = \mathrm{Tr}(\varphi(a)\varphi(b)) = \mathrm{Tr}(\varphi(b)\varphi(a)) = \mathrm{Tr}(\varphi(ba)) =  \chi(ba).\]
To show (b), let $r,s>0$ be the minimal integers such that $a^r=a^{r+s}$.  Write $s=p^kt$ where $p^k$ and $t$ are coprime. If $p=0$, we take $k=0$ and interpret $p^k=1$ and $t=s$ (to avoid having to write out two separate cases).
Then
\begin{equation*}
\varphi(a)^{r} = \varphi(a)^{r+s}
\Longrightarrow \varphi(a)^{r}(\varphi(a)^{s}-1)=0
\Longrightarrow \varphi(a)^{r}(\varphi(a)^{t}-1)^{p^k}=0.
   \end{equation*}
 Let $p(x) = x^r(x^t-1)^{p^k}$.  Then $p(\varphi(a))=0$ and so the minimal polynomial of $\varphi(a)$ divides $p(x)$. Let $\bar{K}$ be an algebraic closure of $K$ containing $K(\xi_n)$. Then in $M_\ell(\bar{K})$ we have
 \[\varphi(a)\sim \begin{bmatrix}
 \lambda_1 & \ast&\cdots &  \ast  \\
                  0 & \lambda_2 &\ddots&\vdots \\
                   \vdots& \ddots &  \ddots &\ast\\
                   0&\cdots& 0&\lambda_{\ell}\\
 \end{bmatrix},\]
 where the non-zero elements among $\lambda_1, \ldots,\lambda_{\ell}\in \bar{K}$ are roots of $x^t-1$. Let $z\ge r$ such that $z\equiv 1\bmod t$ and $z\equiv 0 \bmod {p^k}$.
 Then $a^z=a_{p'}^{\omega+1}$. Therefore,
 \[\varphi(a_{p'}^{\omega+1}) = \varphi(a^z) = \varphi(a)^z \sim
  \begin{bmatrix}
 \lambda_1^z & \ast&\cdots &  \ast  \\
                  0 & \lambda_2^z &\ddots&\vdots \\
                   \vdots& \ddots &  \ddots &\ast\\
                   0&\cdots& 0&\lambda_{\ell}^z\\
 \end{bmatrix}.\]
Observe that, if $\lambda_i\neq 0$, then $\lambda_i^t=1$ and so $\lambda_i^z =\lambda_i$ because $z\equiv 1\bmod t$. Of course, $\lambda_i^z=\lambda_i$ is also true if $\lambda_i=0$.  Therefore,
 \[\chi(a) = \mathrm{Tr}(\varphi(a)) = \sum_{i=1}^{\ell} \lambda_i= \sum_{i=1}^{\ell} \lambda_i^z =
 \mathrm{Tr}(\varphi(a^z)) =  \chi(a_{p'}^{\omega+1}), \]
 which establishes (b).

 Now let $n$ be the $\mathrm{lcm}$ of the orders of the $p$-regular elements of $M$,  $T=\mathop{\mathrm{Im}} \mathrm{Gal}(K(\xi_n)/K)\le \mathbb Z_n^*$, and $L=K(\xi_n)$.
 If $a$ is $p$-regular, then $a=a^{t+1}$ with $t=|a|$ coprime to $p$. Thus  $p(\varphi(a))=0$, for $p(x) = x(x^t-1)$, and $t$ divides $n$. Consequently all eigenvalues of
 $\varphi(a)$ (over $\bar{K}$) belong to $L$. Then over $M_{\ell}(L)$ we have
 \[\varphi(a)\sim \begin{bmatrix}
 \lambda_1 & \ast&\cdots &  \ast  \\
                  0 & \lambda_2 &\ddots&\vdots \\
                   \vdots& \ddots &  \ddots &\ast\\
                   0&\cdots& 0&\lambda_{\ell}\\
 \end{bmatrix} \quad\text{and}\quad \varphi(a^j)\sim \begin{bmatrix}
 \lambda_1^j & \ast&\cdots &  \ast  \\
                  0 & \lambda_2^j &\ddots&\vdots \\
                   \vdots& \ddots &  \ddots &\ast\\
                   0&\cdots& 0&\lambda_{\ell}^j\\
 \end{bmatrix}.\]

We define $\alpha\colon \mathrm{Gal}(L/K)) \to \mathrm{Aut}(M_{\ell}(L))$ by
 \[\alpha(g)\left(\begin{bmatrix}
 a_{11} & \ldots & a_{1\ell}\\
 \vdots & & \vdots \\
  a_{\ell1} & \ldots & a_{\ell\ell}\\
  \end{bmatrix}\right) =
  \begin{bmatrix}
 g(a_{11}) & \ldots & g(a_{1\ell})\\
 \vdots & & \vdots \\
  g(a_{\ell1}) & \ldots & g(a_{\ell\ell})\\
  \end{bmatrix}.
 \]
We note that $M_{\ell}(K)$ is the set of fixed points of  $\mathrm{Gal}(L/K)$ acting on $M_{\ell}(L)$ and hence $\alpha(g)(\varphi(a))=\varphi(a)$ for all $g\in \mathrm{Gal}(L/K)$.  Let $j\in T$ and let $g\in \mathrm{Gal}(L/K)$ be such that $g(\xi_n)=\xi^j$.  Note that if $\lambda$ is either $0$ or an $n^{th}$-root of unity, then $g(\lambda)=\lambda^j$.  Also note that if $A$ and $B$ are similar matrices, then so are $\alpha(g)(A)$ and $\alpha(g)(B)$.
Thus we have
\[\varphi(a)=\alpha(g)(\varphi(a)) \sim
\begin{bmatrix}
 g(\lambda_1) & \ast&\cdots &  \ast  \\
                  0 & g(\lambda_2) &\ddots&\vdots \\
                   \vdots& \ddots &  \ddots &\ast\\
                   0&\cdots& 0&g(\lambda_{\ell})\\
 \end{bmatrix}
 = \begin{bmatrix}
 \lambda_1^j & \ast&\cdots &  \ast  \\
                  0 & \lambda_2^j &\ddots&\vdots \\
                   \vdots& \ddots &  \ddots &\ast\\
                   0&\cdots& 0&\lambda_{\ell}^j\\
\end{bmatrix}
\]
 and so $\chi(a)=\mathrm{Tr}(\varphi(a)) = \sum_{i=1}^{\ell}\lambda_i^j = \mathrm{Tr}(\varphi(a^j)) = \chi(a^j)$, as required.
\end{proof}

The following theorem is the main result of this paper.

\begin{theorem}\label{conj3}
Let $M$ be a finite monoid, $a,b \in M$ and $K$ a field of characteristic $p\geq 0$. Let $T=\mathop{\mathrm{Im}} \mathrm{Gal}(K(\xi_n)/K)\leq \mathbb Z_n^*$ with $n$ the least common multiple of the orders of $p$-regular elements of $M$. Then the following conditions are equivalent:
\begin{enumerate}[a)]
\item $a$ and $b$ are character equivalent over $K$;
\item  $a$ and $b$ are equivalent by $\equiv$, where $\equiv$ is the least equivalence relation in $M$ satisfying:
\begin{enumerate}[i)]
\item $m_1m_2 \equiv m_2m_1, \forall \, m_1, m_2 \in M$
\item $m \equiv m^{\omega+1}_{p'}, \forall \,  m \in M$
\item if $m\in M$ is $p$-regular and $j \in T$, then $m\equiv m^j$;
\end{enumerate}
\item  $a$ and $b$ are equivalent by $\approx$, where $a \approx b$ if and only there exist $x, y \in M$ and $j \in T$ such that
\begin{enumerate}[i)]
\item $xyx=x, yxy=y$
\item $xa^{\omega+1}_{p'}y=(b^{\omega+1}_{p'})^j$
\item $xy=b^{\omega},\, yx=a^{\omega}$.
\end{enumerate}
\end{enumerate}
Moreover, the distinct irreducible characters of $M$ form a basis for the space of $K$-valued functions on $M$ that are constant on $\equiv$-classes.
\end{theorem}
\begin{proof}
First we prove that $\approx$ is an equivalence relation.  It is reflexive, since given $a \in M$, it is enough to choose $x=y=a^{\omega}$ and $j=1$.

Now suppose that $a \approx b$ with $x,y,j$ as above and put $g=a^{\omega+1}_{p'}$ and $h=b^{\omega+1}_{p'}$.  Let $j'$ be the inverse of $j$ in the group $T\leq \mathbb Z_n^*$.  Since the order of $h$ divides $n$ by choice of $n$, we have $h^{jj'}=h$.  Thus, using Lemma~\ref{conjugationlemma}, we have that
\[yhx=yh^{jj'}x=y(xgy)^{j'}x=yxg^{j'}yx=a^{\omega}g^{j'}a^{\omega}=g^{j'}\]  and so  $b\approx a$.

Finally, if $a \approx b$ and $b \approx c$, there exist $x_1, x_2, y_1, y_2 \in M$ and $j_1, j_2 \in T$ such that:
\begin{enumerate}[a)]
\item $x_iy_ix_i=x_i, y_ix_iy_i=y_i$ for $i=1,2$;
\item $x_1a^{\omega+1}_{p'}y_1=(b^{\omega+1}_{p'})^{j_1}$;
\item $x_2b^{\omega+1}_{p'}y_2=(c^{\omega+1}_{p'})^{j_2}$;
\item $y_1x_1=a^{\omega},\, x_1y_1=b^{\omega}$;
\item $y_2x_2=b^{\omega},\, x_2y_2=c^{\omega}$.
\end{enumerate}
By choosing $x_3=x_2x_1$ and $y_3=y_1y_2 \in M$, we obtain \[x_3=x_2x_1=x_2x_1y_1x_1=x_2b^{\omega}x_1=x_2b^{\omega}b^{\omega}x_1=x_2x_1y_1y_2x_2x_1=x_3y_3x_3\] and similarly $y_3x_3y_3=y_3$. Also, we have $y_3x_3=y_1y_2x_2x_1=y_1b^{\omega}x_1=y_1x_1y_1x_1=a^{\omega}$ and $x_3y_3=x_2x_1y_1y_2=x_2b^{\omega}y_2=x_2y_2x_2y_2=c^{\omega}$.

Finally, putting $g=a^{\omega+1}_{p'}$ , $h=b^{\omega+1}_{p'}$ and $k=c^{\omega+1}_{p'}$, we have
\[x_3gy_3=x_2x_1gy_1y_2 = x_2h^{j_1}y_2 = (x_2hy_2)^{j_1}=(k^{j_2})^{j_1}=k^{j_1j_2}\] where we have again used Lemma~\ref{conjugationlemma}.  Thus $a\approx c$.  This concludes the proof that $\approx$ is an equivalence relation. Next we show that it satisfies i)--iii) of b). This will show that b) implies c).

To show i), suppose that $m_1,m_2\in M$.  We want to show that $m_1m_2\approx m_2m_1$.  If $k=|M|!$, then  $(m_1m_2)^{\omega}=(m_1m_2)^k$ and $(m_2m_1)^{\omega}=(m_2m_1)^k$.  Let $x=m_1(m_2m_1)^{2k-1}$ and $y=(m_2m_1)^k m_2$.  Then \[yx=(m_2m_1)^km_2m_1(m_2m_1)^{2k-1}=(m_2m_1)^{3k}=(m_2m_1)^{\omega}\] and
\begin{align*}
xy&=m_1(m_2m_1)^{2k-1}(m_2m_1)^km_2=m_1(m_2m_1)^{3k-1}m_2\\&=m_1m_2(m_1m_2)^{3k-1}=(m_1m_2)^{3k}=(m_1m_2)^{\omega}.
\end{align*}
Clearly then $xyx=m_1(m_2m_1)^{2k-1}(m_2m_1)^{\omega}=m_1(m_2m_1)^{2k-1}=x$ because $(m_2m_1)^{2k-1}\in G_{(m_2m_1)^{\omega}}$.  Also, we have $yxy=(m_2m_1)^{\omega}(m_2m_1)^{\omega}m_2=y$.
We compute that $y(m_1m_2)^{\omega+1}x= (m_2m_1)^km_2(m_1m_2)^{k+1}m_1(m_2m_1)^{2k-1}=(m_2m_1)^{4k+1}=(m_2m_1)^{\omega+1}$.
  By Lemma~\ref{conjugationlemma}, the assignment $z\mapsto yzx$ is a group isomorphism from $G_{(m_1m_2)^{\omega}}$ to $G_{(m_2m_1)^{\omega}}$ with inverse $w\mapsto xwy$.  Thus $y(m_1m_2)^{\omega+1}_{p'}x=(m_2m_1)^{\omega+1}_{p'}$ and so $m_1m_2\approx m_2m_1$ (with $j=1$).

Condition ii) is trivially verified by taking $x=m^{\omega}=y$ and $j=1$ and iii) is trivially satisfied with $x=m^{\omega}=y$.  This completes the proof that b) implies c).

Next we prove that c) implies b).
Let $a, b\in M$ and suppose that $a\approx b$ with $x,y$ as in c).  Then
\begin{equation*}
a \equiv a^{\omega+1}_{p'}=a^{\omega+1}_{p'}a^{\omega}=a^{\omega+1}_{p'}yx \equiv xa^{\omega+1}_{p'}y
=(b^{\omega+1}_{p'})^j \equiv b^{\omega+1}_{p'} \equiv b.
\end{equation*}

That b) implies a) is immediate from Lemma~\ref{characters}.  For the converse, let
\begin{align*}
K^{M/\equiv} &= \{f\colon M/{\equiv}\to K\} \\
&= \{f\colon M\to K\mid a\equiv b\implies f(a) = f(b)\}
\end{align*}
where the last equality is an abuse of notation.
Then $\dim K^{M/{\equiv}} = |M/{\equiv}|$.
If $\chi$ is a character, then $\chi \in K^{M/{\equiv}}$.  Let $\chi_1,\ldots,\chi_r$ be the characters of the inequivalent irreducible representations of $M$.  By Theorem~\ref{indepofchar}, we have that $\chi_1,\ldots,\chi_r$ are linearly independent.
Thus $r\le \dim K^{M/\equiv} = |M/{\equiv}|$.

\begin{claim}
If $r=|M/{\equiv}|$, then $\equiv$ is  character equivalence.
\end{claim}
\begin{proof}
We already have that $a\equiv b$ implies that $a$ and $b$ are character equivalent. Let
$C$ be the $\equiv$-class of the element $a$. Let \[\delta_C(x) = \begin{cases} 1, & \text{if}\ x\in C\\ 0, & \text{if}\ x\not\in C\end{cases}.\] Under the assumption that $r=|M/{\equiv}|$, we have that $\{\chi_1,\ldots,\chi_r\}$ is a basis for $K^{M/{\equiv}}$.
Therefore, $\delta_C = \sum_{i=1}^rk_i\chi_i$ for some $k_i\in K$.  If $a$ and $b$ are character equivalent, then
\[\delta_C(b) = \sum_{i=1}^rk_i\chi_i(b) = \sum_{i=1}^rk_i\chi_i(a) = \delta_C(a)=1. \]
So $b\in C$ and $a\equiv b$.
\end{proof}

We now complete the proof that a) implies b) by showing that $r\geq |M/{\equiv}|$.  This will also show that $\{\chi_1,\ldots,\chi_r\}$ is a basis for $K^{M/{\equiv}}$.  Let $e_1,\ldots,e_k$ be representatives of the $\mathcal D$-classes of idempotents of $M$.

\begin{claim}\label{claim2}
$|M/{\equiv} | \le |\bigsqcup_{i=1}^kp\text{-}\mathrm{reg}(G_{e_i})/{\sim_K}|$
\end{claim}
\begin{proof}
We define a surjective function $f\colon \bigsqcup_{i=1}^k p\text{-}{\mathrm{reg}}(G_{e_i}) \to M/{\equiv}$ such that if $g,h\in p\text{-}\mathrm{reg}(G_{e_i})$ satisfy $g\sim_K h$, then $f(g)=f(h)$.  The claim will then follow.

Let us put $f(g)=[g]_{\equiv}$. Suppose that $g,h\in G_{e_i}$ with $g\sim_K h$.  Then, applying Lemma~\ref{galoisres}, there exist $x\in G_{e_i}$ and $j\in T$ such that $xgx^{-1}=h^j$.  Putting $y=x^{-1}$, we have that $xyx=x$, $yxy=y$, $xy=e_i=yx$ and $xgy=h^j$ implying that $g\equiv h$.
To show that $f$ is surjective, let $m\in M$.   Then $m_{p'}^{\omega+1}\in p\text{-}{\mathrm{reg}}(G_{m^{\omega}})$.
There exist $i$ such that $m^{\omega}\mathrel{\mathcal{D}}e_i$, and $x,y\in M$ such that $xyx=x, yxy=y, xy=e_i, yx=m^{\omega}$.  By Lemma~\ref{conjugationlemma} there
is an isomorphism $\varphi\colon G_{m^{\omega}}\to G_{e_i}$ with $\varphi(a) = xay$.  So $xm_{p'}^{\omega+1}y=\varphi(m_{p'}^{\omega +1})\in p\text{-}\mathrm{reg}(G_{e_i})$ and
$[m]_{\equiv}= [xm_{p'}^{\omega +1}y]_{\equiv} = f(xm_{p'}^{\omega +1}y)$. This establishes that $f$ is surjective, completing the proof of the claim.
\end{proof}

In light of Claim~\ref{claim2}, Theorem~\ref{Berman} and Theorem~\ref{indcar}, we have
\begin{equation*}
|M/{\equiv}|\leq |\bigsqcup_{i=1}^kp\text{-}\mathrm{reg}(G_{e_i})/{\sim_K}| = \sum_{i=1}^k|\Irr_K(G_{e_i})| = |\Irr_K(M)|=r
\end{equation*}
as required.  This completes the proof of the theorem.
\end{proof}

Let us specialize the result to the case that $K$ is algebraically closed.  In this case $K(\xi_n)=K$ and so $T=\{1\}$.  Thus we recover McAlister's result~\cite[Theorem~2.2]{McAlisterCharacter} and obtain its analogue in positive characteristic.

\begin{cor}
Let $M$ be a finite monoid, $a,b \in M$ and $K$ an algebraically closed field.
\begin{enumerate}[1)]
\item If the characteristic of $K$ is $0$, then the following are equivalent:
\begin{enumerate}[a)]
\item $a$ and $b$ are character equivalent over $K$;
\item there exist $x, y \in M$ such that
\begin{enumerate}[i)]
\item $xyx=x, yxy=y$
\item $xa^{\omega+1}y=b^{\omega+1}$
\item $xy=b^{\omega},\, yx=a^{\omega}$.
\end{enumerate}
\end{enumerate}
\item If the characteristic of $K$ is $p>0$, then the following are equivalent:
\begin{enumerate}[a)]
\item $a$ and $b$ are character equivalent over $K$;
\item there exist $x, y \in M$ such that
\begin{enumerate}[i)]
\item $xyx=x, yxy=y$
\item $xa^{\omega+1}_{p'}y=b^{\omega+1}_{p'}$
\item $xy=b^{\omega},\, yx=a^{\omega}$.
\end{enumerate}
\end{enumerate}
\end{enumerate}
\end{cor}

Let us next specialize to a finite field $\mathbb F_q$ of order $q$.  Then $\mathrm{Gal}(\mathbb F_q(\xi_n)/\mathbb F_q)$ is generated by the Frobenius automorphism $c\mapsto c^q$.  Thus $T$ is the subgroup of $\mathbb Z_n^*$ generated by $q$ and so we have the following corollary.

\begin{cor}
Let $M$ be a finite monoid, $a,b \in M$ and $\mathbb F_q$ be a field of order $q$ and characteristic $p$. Then the following conditions are equivalent:
\begin{enumerate}[a)]
\item $a$ and $b$ are character equivalent over $\mathbb F_q$;
\item  there exist $x, y \in M$ and $k\geq 0$ such that
\begin{enumerate}[i)]
\item $xyx=x, yxy=y$
\item $xa^{\omega+1}_{p'}y=(b^{\omega+1}_{p'})^{q^k}$
\item $xy=b^{\omega},\, yx=a^{\omega}$.
\end{enumerate}
\end{enumerate}
\end{cor}

Finally, we specialize to the case that $K=\mathbb Q$.  In this case, $T=\mathbb Z_n^*$.  If $g\in M$ is a group element, then since $|g|$ divides $n$, we have that $\{g^j\mid j\in \mathbb Z_n^*\}$ is precisely the set of generators of $\langle g\rangle$.  From this we obtain our next result.

\begin{cor}
Let $M$ be a finite monoid and $a,b \in M$. Then the following conditions are equivalent:
\begin{enumerate}[a)]
\item $a$ and $b$ are character equivalent over $\mathbb Q$;
\item there exist $x, y \in M$ such that
\begin{enumerate}[i)]
\item $xyx=x, yxy=y$
\item $x\langle a^{\omega+1}\rangle y=\langle b^{\omega+1}\rangle$
\item $xy=b^{\omega},\, yx=a^{\omega}$.
\end{enumerate}
\end{enumerate}
\end{cor}

We remark that $K^{M/{\equiv}}$ is a ring with pointwise addition and multiplication (in fact, a $K$-algebra) and that the proof of Theorem~\ref{conj3} also establishes the following result.

\begin{cor}\label{charringiso}
Let $M$ be a finite monoid and $K$ a field.  Let $e_1,\ldots, e_k$ be representatives of the $\mathcal D$-classes of idempotents of $M$.  Then
the map \[\psi\colon \bigsqcup_{i=1}^kp\text{-}\mathrm{reg}(G_{e_i})/{\sim_K}\to M/{\equiv}\] given by $\psi([g]_{\sim_K}) = [g]_{\equiv}$ is a bijection and hence induces a ring isomorphism \[\Psi\colon K^{M/{\equiv}}\to \prod_{i=1}^kK^{p\text{-}\mathrm{reg}(G_{e_i})/{\sim_K}}\] given by \[\Psi(f)= \left(f\psi_{|_{p\text{-}\mathrm{reg}(G_{e_1})/{\sim_K}}},\ldots,f\psi_{|_{p\text{-}\mathrm{reg}(G_{e_k})/{\sim_K}}}\right)\] for $f\in K^{M/{\equiv}}$.
\end{cor}

\section{Virtual characters}
A mapping $f\colon M\to K$ in $K^{M/{\equiv}}$ is called a \emph{virtual character} if $f$ is a $\mathbb Z$-linear combination of irreducible characters or equivalently $f=\chi_1-\chi_2$ where $\chi_1,\chi_2$ are characters of $M$.  Note that the virtual characters form a subring of $K^{M/{\equiv}}$ because the product of two characters is the character of the tensor product of the corresponding representations~\cite{McAlisterCharacter,RhodesZalc}.

We will need a characterization of virtual characters in Section~\ref{s:application}. This extends results of McAlister~\cite{McAlisterCharacter} proven for the field of complex numbers.

\begin{lemma}\label{restricttochar}
Let $\chi$ be a character of a monoid $M$ and let $e\in E(M)$.  Then $\chi_{|_{G_e}}$ is either identically zero or a character of $G_e$.
\end{lemma}
\begin{proof}
Let $\rho\colon M\to M_r(K)$ be a representation whose character is $\chi$.  Then $\rho(e)$ is an idempotent matrix and hence we have an internal direct sum decomposition $K^r=\mathop{\mathrm{Im}}\rho(e)\oplus\ker \rho(e)$.  Clearly, $G_e=eG_ee$ acts by automorphisms on $\mathop{\mathrm{Im}}\rho(e)$ and annihilates $\ker \rho(e)$.  Thus \[\rho_{|_{G_{e}}} \sim
\begin{bmatrix}
\hat{\rho} & 0\\
0 & 0\\
\end{bmatrix}\]
where $\hat{\rho}$ is a representation of $G_e$ and so $\chi_{|_{G_e}}=\chi_{\hat{\rho}}$ is a character.
\end{proof}

We can now characterize the ring of virtual characters.  If $M$ is a monoid, let $\ch_K(M)$ denote the ring of virtual characters of $M$ (over $K$). The above lemma shows that there is a natural ring homomorphism $\rho_e\colon \ch_K(M)\to \ch_K(G_e)$ given by restriction for each $e\in E(M)$.

\begin{theorem}\label{virtualchars}
Let $M$ be a finite monoid and $K$ a field.  Let $e_1,\ldots, e_k$ be representatives of the $\mathcal D$-classes of idempotents of $M$.  Then $f\in K^{M/{\equiv}}$ is a virtual character if and only if $f_{|_{G_{e_i}}}$ is a virtual character for $i=1,\ldots, k$.  Moreover, there is a ring isomorphism
\[\Phi\colon \ch_K(M)\to \prod_{i=1}^k\ch_K(G_{e_i})\] given by \[\Phi(f) = (f_{|_{G_{e_1}}},\ldots,f_{|_{G_{e_k}}})\] for $f\in \ch_K(M)$.
\end{theorem}
\begin{proof}
Lemma~\ref{restricttochar} implies that the restriction of a virtual character of $M$ to $G_{e_i}$ is a virtual character and hence  $\Phi$ is a  well-defined homomorphism.  We prove that $\Phi$ is surjective by induction using the order $\preceq$ defined after Lemma~\ref{conjugationlemma}.  More precisely, order $e_1,\ldots,e_k$ so that $e_i\preceq e_j$ implies $i\leq j$.  First note that in a finite monoid $M$ with identity $1$, one has that $G_1$ is the group of invertible elements of $M$ and that $1\in MmM$ if and only if $m\in G_1$ (cf.~\cite[Proposition~1.2]{TilsonXI}). In particular, $M\setminus G_1$ is closed under multiplication and $1$ is not $\mathcal D$-equivalent to any other idempotent.  Also $Me_jM\subseteq M=M1M$ and so $1=e_k$.

If $\varphi\in \Irr_K(G_{e_k})=\Irr_K(G_1)$, then $\varphi$ extends to an irreducible representation $\varphi'$ of $M$ by putting  \[\varphi'(m)=\begin{cases}\varphi(m), & \text{if}\ m\in G_{e_k}\\ 0, & \text{else.}\end{cases}\] Then $\Phi(\chi_{\varphi'}) = (0,0,\ldots, 0,\chi_{\varphi})$.  It follows that $\{0\}\times\cdots\times \{0\}\times \ch_K(G_{e_k})$ is in the image of $\Phi$.

Assume that $\{0\}\times \cdots \times \ch_K(G_{e_{i+1}})\times\cdots\times \ch_K(G_{e_k})$ is in the image of $\Phi$ by induction.  Let $\varphi\in \Irr_K(G_{e_i})$.  Then by Theorem~\ref{indcar} and Corollary~\ref{rescar}, there is an irreducible representation $\varphi'$ of $M$ whose character $\chi$ satisfies $\chi_{|_{G_{e_i}}}=\chi_{\varphi}$ and $\chi(m)=0$ if $e_i\notin MmM$. It follows that $\Phi(\chi)=(0,\ldots,0,\chi_{\varphi_i},f_{i+1},\ldots, f_k)$ where $f_j\in \ch_K(G_{e_j})$ for $i+1\leq j\leq k$.  But then using the induction hypothesis, we deduce $(0,\ldots,0,\chi_{\varphi_i},0,\ldots, 0)$ is in the image of $\Phi$ and so $\{0\}\times \cdots \times \ch_K(G_{e_{i}})\times\cdots\times \ch_K(G_{e_k})$ is in the image of $\Phi$.  We conclude by induction that $\Phi$ is surjective.

Injectivity of $\Phi$ follows from the injectivity of $\Psi$ in Corollary~\ref{charringiso}.

Suppose that $f\in K^{M/{\equiv}}$ satisfies $f_{|_{G_{e_i}}}\in \ch_K(G_{e_i})$ for $i=1,\ldots, k$. Write $f=c_1\chi_1+\cdots+c_s\chi_s$ with $\chi_1,\ldots, \chi_s$ irreducible characters of $M$ and $c_1,\ldots, c_s\in K\setminus \{0\}$.  We prove by induction on $s$ that $f\in \ch_K(M)$.  If $s=0$, there is nothing to prove.

Let $e_i$ be a minimal (with respect to $\preceq$) element among the apexes of $\chi_1,\ldots, \chi_s$.  Without loss of generality, we may assume that $\chi_1,\ldots, \chi_t$ have apex $e_i$ and $\chi_{t+1},\ldots, \chi_s$ have apex different than $e_i$.  Then, as in the proof of Theorem~\ref{indepofchar}, we have that $\chi_j(G_{e_i}) = 0$ if $j>t$ and that $\chi_1,\ldots,\chi_t$ restrict to distinct irreducible characters of $G_{e_i}$.  Since $f_{|_{G_{e_i}}}$ is a virtual character, we deduce by the linear independence of the irreducible characters of $G_{e_i}$ over $K$ that $c_1,\ldots,c_t\in \mathbb Z\cdot 1$. Moreover, $c_{t+1}\chi_{t+1}+\cdots+c_s\chi_s=f-(c_1\chi_1+\cdots+\cdot +c_t\chi_t)$ still restricts to a virtual character at each $G_{e_j}$ because $f$ does, each $\chi_j$ does (by Lemma~\ref{restricttochar}) and $c_1,\ldots, c_t\in \mathbb Z\cdot 1$.  Thus $f-(c_1\chi_1+\cdots+\cdot +c_t\chi_t)$ is a virtual character by induction, and hence so is $f$ because $c_1,\ldots, c_t\in \mathbb Z\cdot 1$.
\end{proof}

\section{An application: a theorem of Berstel and Reutenauer}\label{s:application}
Let $A$ be a finite set and $A^*$ the free monoid on $A$, that is, the set of all words in the alphabet $A$.  The empty word will be denoted $1$.  A subset $L\subseteq A^*$ is often called a (formal) \emph{language}.  In automata theory, a language $L\subseteq A^*$ is called \emph{regular} (or \emph{rational}) if it is accepted by a finite state automaton or equivalently if there is a finite monoid $M$ and a surjective monoid homomorphism $\eta\colon A^*\to M$ such that $L=\eta^{-1}(\eta(L))$~\cite{EilenbergA,Eilenberg}.  The \emph{zeta function} of $L$ is
\[\zeta_L(x) = \exp\left({\sum_{n=1}^{\infty}\frac{a_n}{n}x^n}\right)\] where $a_n$ is the number of words in $L$ of length $n$; see~\cite{BerstelReutenauer,BRzeta}. The length of a word $w$ is denoted $|w|$.

Let $K$ be a field and let $K{\langle A\rangle}$ be the ring of polynomials in non-commuting variables $A$ with coefficients in $K$, that is, the free $K$-algebra on $A$.  Let $K{\llangle A\rrangle}$ be the ring of formal power series in non-commuting variables $A$ with coefficients in $K$.  A power series is \emph{rational} if it belongs to the smallest $K$-subalgebra of $K{\llangle A\rrangle}$ containing the polynomials and closed under inversion of power series with non-zero coefficient of $1$.

One has a right $K{\langle A\rangle}$-module structure on $K{\llangle A\rrangle}$ by defining
\[\left(\sum_{w\in A^*}c_ww\right)\cdot a = \sum_{w\in A^*}c_w wa\] for $a\in A$.  A celebrated theorem of Sch\"utzenberger states that a power series $f$ is rational if and only if the $K{\langle A\rangle}$-submodule generated by $f$ is finite dimensional over $K$~\cite{BerstelReutenauer}.  Perrin defines a power series to be \emph{completely reducible} if the representation of $A^*$ associated to the $K{\langle A\rangle}$-submodule generated by $f$ is a direct sum of irreducible representations~\cite{Perrincomplred}.  The completely reducible series form a $K{\langle A\rangle}$-submodule.

Let us say that a power series $f\in K{\llangle A\rrangle}$ is a \emph{trace series} if there is a character $\chi\colon A^*\to K$ of a finite dimensional representation of $A^*$ over $K$ such that $f=\sum_{w\in A^*}\chi(w)\cdot w$.
Perrin observes that trace series, and hence linear combinations of trace series,  are completely reducible~\cite{Perrincomplred}; see also~\cite{BerstelReutenauer}.

If $L\subseteq A^*$ is a language, then its \emph{characteristic series} is \[f_L=\sum_{w\in L}w\in K{\llangle A\rrangle}.\] If $L$ is regular, then it is well known that $f_L$ is rational~\cite{BerstelReutenauer}.  A language $L$ is said to be \emph{cyclic} if:
\begin{enumerate}[i)]
\item $u\in L\iff u^s\in L$ for all $s>0$
\item $uv\in L\iff vu\in L$.
\end{enumerate}

The key example is the following. Let $\mathcal X\subseteq A^{\mathbb Z}$ be a symbolic dynamical system, that is, a non-empty closed subspace (in the product topology) invariant under the shift map $\sigma\colon A^{\mathbb Z}\to A^{\mathbb Z}$ defined by \[\sigma(\cdots a_{-2}a_{-1}.a_0a_1\cdots)= \cdots a_{-2}a_{-1}a_0.a_1\cdots\] (see~\cite{MarcusandLind} for background on symbolic dynamics, including notation).  Let $L$ be the set of all words $w\in A^*$ such that $\cdots ww.ww\cdots$ is a periodic point of $\mathcal X$. Then $L$ is a cyclic language.  If $\mathcal X$ is a shift of finite type, or more generally a sofic shift, then $L$ will be regular.  The zeta function for $L$ gives the zeta function of the shift, and the rationality of the zeta function of a sofic shift follows from the rationality of the zeta function of any regular cyclic language~\cite{BRzeta,BerstelReutenauer}.

Berstel and Reutenauer~\cite{BRzeta,BerstelReutenauer} proved that if $L$ is a regular cyclic language, then the characteristic series $f_L$ of $L$ is a $\mathbb Z$-linear combination of trace series over any field $K$, and hence is completely reducible.  They also used this to prove the rationality of $\zeta_L$. We show that this result is an immediate corollary of Theorem~\ref{conj3} and Theorem~\ref{virtualchars}.  Perrin proved, using the results of McAlister~\cite{McAlisterCharacter}  that $f_L$ is a $K$-linear combination of trace series in the case that $K$ is an algebraically closed field of characteristic $0$.

\begin{theorem}[Berstel-Reutenauer~\cite{BRzeta}]\label{reut}
Let $L\subseteq A^*$ be a regular cyclic language and let $K$ be a field.  Then the characteristic series $f_L$ of $L$ is a $\mathbb Z$-linear combination of trace series.
\end{theorem}

Note that since a non-negative integral combination of trace series is again a trace series, the theorem really asserts that $f_L$ is a difference of trace series.

Theorem~\ref{reut} is a straightforward consequence of the following lemma about finite monoids.

\begin{lemma}\label{cyclicsubset}
Let $M$ be a finite monoid and $K$ a field.  Let $X$ be a subset of $M$ such that:
\begin{enumerate}[i)]
\item $m\in X\iff m^s\in X$ for all $s>0$;
\item $m_1m_2\in X \iff m_2m_1\in X$.
\end{enumerate}
Then the characteristic function $I_X\colon M\to K$ of $X$ defined by
\[I_X(m)=\begin{cases} 1, & \text{if}\ m\in X \\ 0, &  \text{if}\ m\not\in X \end{cases}\]
is a virtual character of $M$ over $K$.
\end{lemma}
\begin{proof}
We retain the notation of Theorem~\ref{conj3}. By hypotheses we have, for $m_1,m_2,m\in M$,
\begin{alignat*}2
I_X(m_1m_2) &= I_X(m_2m_1)&\\
I_X(m)&= I_X(m_{p'}^{\omega+1})&\\
I_{X}(m)&=I_X(m^j) & \quad \text{for}\ j\in T.
\end{alignat*}
Therefore,  $I_X\in K^{M/{\equiv}}$  by Theorem~\ref{conj3}.

 If $e\in E(M)$ and $g\in G_e$, then $g^k=e$ for some $k>0$.  Thus $g\in X$ if and only if $e\in X$.  We conclude that either $G_e\subseteq X$ or $G_e\cap X=\emptyset$.  In the latter case, $I_X$ restricts to $0$ on $G_e$ and hence is a virtual character; in the former $I_X$ restricts to the character of the trivial representation of $G_e$.  We conclude that $I_X$ is a virtual character by Theorem~\ref{virtualchars}.
\end{proof}

We can now prove the theorem of Berstel and Reutenauer.

\begin{proof}[Proof of Theorem~\ref{reut}]
Let $\eta\colon A^*\to M$ be a surjective monoid homomorphism with $L=\eta^{-1}(\eta(L))$ and $M$ finite.   Let $X=\eta(L)$.
 Then
 \begin{equation}\label{equa1}
 f_L=\sum_{w\in L}w=\sum_{w\in A^*}I_X(\eta(w))\cdot w=\sum_{w\in A^*}I_X\circ \eta(w)\cdot w
 \end{equation}
where $I_X$ is as in Lemma~\ref{cyclicsubset}.  Since $L$ is a cyclic language and $L=\eta^{-1}(X)$, one easily checks that $X$ satisfies the hypotheses of Lemma~\ref{cyclicsubset}  and hence $I_X=\chi_1-\chi_2$ where  $\chi_1,\chi_2$ are characters of $M$.  Indeed, we have \[\eta(w)\in X\iff w\in L\iff w^s\in L\iff \eta(w)^s\in X\] for all $s>1$ and $w\in A^*$, and \[\eta(u)\eta(v)\in X\iff uv\in L\iff vu\in L\iff \eta(v)\eta(u)\in X\] for all $u,v\in A^*$.

Observe that if $\chi\colon M\to K$ is a character, then $\chi\circ\eta\colon A^*\to K$ is a character.  Thus $I_X\circ \eta = \chi_1\circ \eta-\chi_2\circ\eta$ and so $f_L$ is a $\mathbb Z$-linear combination of trace series by \eqref{equa1}.
\end{proof}

Because it is so pretty, we recall how Theorem~\ref{reut} implies the rationality of $\zeta_L$ for a regular cyclic language $L$. Also, this will highlight the importance of using virtual characters.

\begin{theorem}[Berstel-Reutenauer~\cite{BRzeta}]
Let $L\subseteq A^*$ be a regular cyclic language.  Then $\zeta_L$ is rational.
\end{theorem}
\begin{proof}
By Theorem~\ref{reut} there exist characters $\chi_1,\chi_2\colon A^*\to \mathbb C$ such that \[f_L = \sum_{w\in L}w= \sum_{w\in A^*}\left(\chi_1(w)-\chi_2(w)\right)w.\]  Let $\varphi_i\colon A^*\to M_{n_i}(\mathbb C)$ be representations with $\chi_i=\chi_{\varphi_i}$, for $i=1,2$.
Let $M_1=\sum_{a\in A}\varphi_1(a)$ and $M_2=\sum_{a\in A}\varphi_2(a)$.  A simple induction argument shows that
\begin{equation}\label{adjmatrix}
M_i^n=\sum_{|w|=n}\varphi_i(w)
\end{equation}
for $i=1,2$.
If $a_n$ is the number of words of length $n$ in $L$, then, for $n\geq 1$,
\begin{align*}
a_n &= \sum_{|w|=n}\chi_1(w)-\chi_2(w)\\ &=\sum_{|w|=n}\mathrm{Tr}(\varphi_1(w))-\mathrm{Tr}(\varphi_2(w))
\\ &= \mathrm{Tr}\left(\sum_{|w|=n}\varphi_1(w)\right)- \mathrm{Tr}\left(\sum_{|w|=n}\varphi_2(w)\right)\\
&= \mathrm{Tr}(M_1^n)-\mathrm{Tr}(M_2^n)
\end{align*}
by \eqref{adjmatrix}.

Therefore, \[\zeta_L(x)=\exp\left(\sum_{n=1}^{\infty}\frac{\mathrm{Tr}(M_1^n)-\mathrm{Tr}(M_2^n)}{n}x^n\right)
=\frac{\exp\left(\sum_{n=1}^{\infty}\frac{\mathrm{Tr}(M_1^n)}{n}x^n\right)}{\exp\left(\sum_{n=1}^{\infty}\frac{\mathrm{Tr}(M_2^n)}{n}x^n\right)}\]
and so it suffices to show that if $B$ is a $k\times k$-matrix over $\mathbb C$, then the series
\[g(x)=\exp\left(\sum_{n=1}^{\infty}\frac{\mathrm{Tr}(B^n)}{n}x^n\right)\] is rational.  But if $\lambda_1,\ldots, \lambda_k$ are the eigenvalues of $B$ with multiplicities, then
\[g(x) = \exp\left(\sum_{n=1}^{\infty}\frac{\lambda_1^n+\cdots+\lambda_k^n}{n}x^n\right)=\prod_{i=1}^k\exp\left(-\log (1-\lambda_ix)\right)=\prod_{i=1}^k\frac{1}{1-\lambda_ix}\]
is rational (in fact, it is $1/\det(1-xB)$).
\end{proof}

\section*{Acknowledgments}
The third author would like to thank Geoff Robinson~\cite{Geoffanswer} for pointing him toward Berman's paper~\cite{Berman}.

\def\malce{\mathbin{\hbox{$\bigcirc$\rlap{\kern-7.75pt\raise0,50pt\hbox{${\tt
  m}$}}}}}\def\cprime{$'$} \def\cprime{$'$} \def\cprime{$'$} \def\cprime{$'$}
  \def\cprime{$'$} \def\cprime{$'$} \def\cprime{$'$} \def\cprime{$'$}
  \def\cprime{$'$} \def\cprime{$'$}

\end{document}